\documentclass[reqno]{amsart}
\usepackage{amsfonts, amssymb, amsmath, amsthm, diagrams}
\usepackage{graphics}
\usepackage{psfrag}
\newcommand{\R}{\mathbb{R}}
\newcommand{\C}{\mathbb{C}}
\newcommand{\h}{\mathbb{H}}
\newcommand{\HP}{\mathbb{HP}}
\newcommand{\cp}{\mathbb{CP}}
\newcommand{\Z}{\mathbb{Z}}
\newcommand{\U}{\mathcal{U}}
\newcommand{\g}{\mathfrak{g}}
\newcommand{\s}{\mathcal{S}}
\newcommand{\mf}[1]{\mathfrak{#1}}
\newcommand{\mc}[1]{\mathcal{#1}}
\newcommand{\mb}[1]{\mathbf{#1}}
\newcommand{\tn}[1]{\textnormal{#1}}
\newcommand{\ms}{M_\s}
\newcommand{\mrs}{M_{\mc{R},\s}}
\newcommand{\bs}{\boldsymbol}
\newtheorem{T}{Theorem}[section]
\newtheorem{TP}[T]{Proposition}
\newtheorem{TL}[T]{Lemma}
\newtheorem{TC}[T]{Corollary}
\theoremstyle{definition} 
\theoremstyle{definition}  \newtheorem{TR}[T]{Remark}
\theoremstyle{definition}  \newtheorem{TD}[T]{Definition}
\theoremstyle{definition}  \newtheorem*{Acknowledgements}{Acknowledgements}
\theoremstyle{plain} \newtheorem*{TA}{Theorem A}
\theoremstyle{plain} \newtheorem*{TB}{Theorem B}

\newarrow {Implies} ===={=>}
\begin{document}
\title{Scalar-flat K\"ahler orbifolds via quaternionic-complex reduction}
\author{Dominic Wright}
\email{dominic.wright@imperial.ac.uk}
\begin{abstract}
We prove that any asymptotically locally Euclidean scalar-flat K\"ahler 4-orbifold whose isometry group contains a 2-torus is isometric, up to an orbifold covering, to a \emph{quaternionic-complex quotient} of a $k$-dimensional quaternionic vector space by a $(k-1)$-torus. In order to do so, we first prove that any compact anti-self-dual 4-orbifold with positive Euler characteristic whose isometry group contains a 2-torus is conformally equivalent, up to an orbifold covering, to a \emph{quaternionic quotient} of $k$-dimensional quaternionic projective space by a $(k-1)$-torus.
\end{abstract}
\maketitle
\section{Introduction}\label{Introduction}
The hyperK\"ahler quotient construction (Hitchin et al. \cite{HKLR87}) has been successively generalized to the quaternionic-K\"ahler quotient construction (Galicki-Lawson \cite{GL88}) and to the quaternionic quotient construction (Joyce \cite{J91}). In four dimensions a quaternionic structure is defined to be an anti-self-dual conformal structure. This enabled Joyce, in \cite{J92}, to use the quaternionic quotient to glue copies of the weighted complex projective plane, thus constructing anti-self-dual metrics on $\overline{\cp}^2\#\ldots\#\overline{\cp}^2$. This preprint led Joyce to an explicit construction of these metrics in \cite{J95}. It was then shown by Fujiki \cite{F00} that these were the only examples of torus symmetric anti-self-dual metrics on compact 4-manifolds with positive orbifold Euler characteristic. The construction of these metrics naturally generalizes to orbifolds and the extension of Fujiki's classification to this setting was confirmed in Theorem A of \cite{W09a}. In this article we return to Joyce's starting point, the quaternionic quotient, in order to prove the following theorem.

\begin{TA}\label{Theorem: Theorem A}
Let $(M,[g])$ be a compact anti-self-dual 4-orbifold with positive orbifold Euler characteristic whose isometry group contains a 2-torus. Then there is a conformal isometry between the universal cover of $(M,[g])$ and a quaternionic quotient of $\HP^{k}$ by a $(k-1)$-dimensional subtorus of the maximal torus in $Sp(1).Sp(k)$, where $k - 1 = b_2(M)$.
\end{TA}

Theorem A can be viewed as a generalization of Theorem A in \cite{CS06} and our proof, rather than extending Joyce's gluing technique, is instead based upon the construction of self-dual Einstein orbifolds via quaternionic-K\"ahler reduction, cf. \cite{BGMR98}, \cite{CS06} and \cite{GL88}. The significant difference between our proof and the quaternionic-K\"ahler approach is that we have a less restrictive definition of the moment map. As a corollary of the construction the twistor space of an orbifold satisfying Theorem A can be realized as the GIT quotient of the intersection of $k-2$ quadrics in $\cp^{2k-1}$ by a complex $(k-2)$-torus. It follows that these twistor spaces are \emph{Moishezon}.

The hyperK\"ahler quotient was used by Kronheimer in \cite{K89a} and \cite{K89b} to construct and classify asymptotically locally Euclidean (ALE) hyperK\"ahler 4-manifolds. We will be interested in the more general case of ALE K\"ahler 4-manifolds with \emph{zero scalar curvature} rather than zero Ricci curvature. In the scalar-flat K\"ahler case we can no longer choose three complex structures globally. However, they do exist locally; more precisely, there is a quaternionic structure. There is also a distinguished choice of complex structure coming from the K\"ahler structure. This provides a 4-dimensional example of what is more generally referred to as a \emph{quaternionic-complex structure}: a quaternionic structure that contains a globally defined complex structure. In \cite{J91} Joyce provides a quotient construction for quaternionic-complex manifolds. We use this, together with Theorem A, to prove the following:
\begin{TB}\label{Theorem: Theorem B}
Let $X$ be an ALE scalar-flat K\"ahler 4-orbifold whose isometry group contains a 2-torus. Then the universal cover of $X$ is isometric to a quaternionic-complex quotient of $\h^{k}$ by a $(k-1)$-dimensional subtorus of the maximal torus in $Sp(1).Sp(k)$, where $k-1 = b_2(X)$.
\end{TB}

In the smooth case, surfaces satisfying Theorem B arise naturally when considering extremal K\"ahler metrics, such as in \cite{CLW07} and \cite{RS05}. By Corollary 1.1 of \cite{W09a} these surfaces belong to the class constructed by Calderbank and Singer in \cite{CS04} using techniques developed by Joyce in \cite{J95}. It was noted in \cite{CS04} that these surfaces are diffeomorphic to toric resolutions of $\C^2/\Gamma$, for some cyclic $\Gamma\subset U(2)$. Analogously, the ALE hyperK\"ahler 4-manifolds constructed by Kronheimer are diffeomorphic to minimal resolutions of $\C^2/\Gamma$, where $\Gamma$ is a finite subgroup of $SU(2)$. Furthermore, on the minimal resolutions of the $A_k$-singularities, which correspond to the cyclic subgroups in $SU(2)$, the quaternionic-complex and hyperK\"ahler quotient constructions coincide.

This article is arranged as follows. In Section \ref{Section: Definitions} we first review quaternionic geometry and the quaternionic quotient. Then, we give a description of compact toric 4-orbifolds in terms of combinatorial data and finally, we outline the classification of orbifolds satisfying Theorem A in terms of this combinatorial data and an invariant of the conformal structure.

In Section \ref{Section: Quotient} we construct a family of toric anti-self-dual orbifolds using the quaternionic quotient, which generalizes the construction of self-dual Einstein orbifolds using the quaternionic-K\"ahler quotient \cite{BGMR98}. To do this we first define a $(k-1)$-torus action on $\HP^{k}$ using the combinatorial data associated with a compact toric orbifold. This torus action determines a family of moment maps and, in order to take the quotient, we choose a subset of these moment maps for which the zero sets intersect the complexified torus orbits transversally.

The classification in \cite{W09a} determines an anti-self-dual orbifold in terms of meromorphic functions associated with its twistor space. We use these meromorphic functions in Section \ref{Section: Proofs of Theorem A and B}, to show that the moment map can be chosen so that the quotient is conformally equivalent to any anti-self-dual structure on the underlying compact toric orbifold. This completes the proof of Theorem A, and Theorem B readily follows. To conclude, we make some remarks about deformations, within the framework of the quaternionic quotient, of the orbifolds we have constructed.

\begin{Acknowledgements}
This work forms part of the author's PhD thesis at Imperial College London, funded by the Engineering and Physical Sciences Research council. I would like to thank my thesis advisor Simon Donaldson for helpful discussions during the preparation of this article.
\end{Acknowledgements}

\section{Preliminary material}\label{Section: Definitions}
\subsection{Hypercomplex and quaternionic structures}\label{Subsection: Hypercomplex and quaternionic}
Let $M$ be a smooth $4k$-dimensional manifold. A \emph{hypercomplex structure} on $M$ consists of \emph{three complex structures} $I_1,I_2,I_3\in \tn{End}(TM)$ that satisfy the defining relations of a quaternionic algebra:
$$I_1I_2I_3 = -\mathbb{I},$$
where $\mathbb{I}$ denotes the identity endomorphism \cite{B88}. The local model for a hypercomplex manifold is the quaternionic vector space $$\h^k = \{(q_1, \ldots, q_k) : q_i \in \h\},$$
where $I_1$, $I_2$ and $I_3$ correspond to left multiplication by $i$, $j$ and $k$, respectively.
The quotient of $\h^k\backslash\{0\}$ with respect to the action of $GL(1,\h)\cong \h^*$ by left multiplication defines the manifold
$$\HP^{k-1} = \{[q_1: \ldots: q_k] : q_i \in \h\}.$$
Although the complex structures $I_1, I_2, I_3$ do not in general descend to the quotient, the subbundle they span does. Thus, the hypercomplex structure on $\h^k$ defines a subbundle $H \subset \tn{End}(T\HP^{k-1})$, which is pointwise isomorphic to $\tn{Im}\h$ as an algebra. This structure on $\HP^{k-1}$ provides a local model for our next definition.

Let $M$ be a smooth $4(k-1)$-dimensional manifold for $k>2$. A \emph{quaternionic structure} on $M$ is defined to be a subbundle $H \subset \tn{End}(TM)$ that is pointwise isomorphic to $\tn{Im}\h$ as an algebra and admits a torsion free connection preserving $H$. The group of automorphisms of the tangent plane preserving the quaternionic structure is $GL(n,\h).GL(1,\h)$, where $GL(1,\h)$ acts on the complex structures in $H$ and $GL(n,\h)$ preserves them. This group has double cover $GL(n,\h)\times GL(1,\h)$ where the deck transformation group is generated by $(-1,-1)$.

Let $V$ be the natural representation of $GL(1,\h)$ acting on the left, with $GL(n,\h)$ acting trivially. Locally there exists a $V$ bundle associated with the locally defined $GL(n,\h)\times GL(1,\h)$ principal bundle. Since there does not necessarily exist a globally defined $GL(n,\h)\times GL(1,\h)$ bundle, we take the quotient by $\{\pm 1\} \subset GL(1,\h)$ to form a bundle associated with the globally defined $GL(n,\h).GL(1,\h)$ principal bundle. This $\h/\{\pm 1\} \cong CO(3)$ bundle is referred to as the associated bundle of $M$ and is denoted by $\mc{U}(M)$. It was proven by Salamon in \cite{S86} that the total space of $\U(M)$ has a natural hypercomplex structure.

The coordinates on $\h^{k}$ can be written as $(x_1 + y_1j,\ldots, x_k + y_kj)$, where the $x_l$ and $y_l$ are complex coordinates with respect to $I_1$. The projection map $\pi_1 : \h^k \rightarrow \HP^{k-1}$ factors through the projection
$$\pi_2:\h^k \rightarrow \cp^{2k-1}; (x_1 + y_1j,\ldots, x_k + y_kj)\mapsto [x_1: y_1:\ldots:x_k:y_k].$$ Therefore, the complex manifold $\cp^{2k-1}$ is  the total space of the fibration $\pi = \pi_1\circ\pi_2^{-1}: \cp^{2k-1} \rightarrow \HP^{k-1}$.

\begin{figure}[htbp]
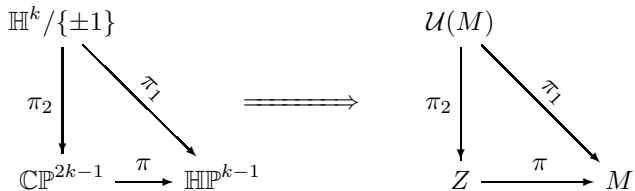

\begin{diagram}
\h^k/\{\pm 1\}   &                         &                  &               &           &     \U(M)            &                         &                \\
\dTo^{\pi_2}     &      \rdTo^{\pi_1}        &      \quad       &   \rImplies   &   \quad   &     \dTo^{\pi_2}     &      \rdTo^{\pi_1}&                        \\
\cp^{2k-1}       &      \rTo^{\pi}       &      \HP^{k-1}   &               &           &     Z                &      \rTo^{\pi}       &      M         \\
\end{diagram}
  \caption{The local model for twistor spaces and associated bundles}
\end{figure}

The fibres of $\pi$ are lines in $\cp^{2k-1}$. These fibres are preserved by the fixed point free anti-holomorphic involution on $\cp^{2k-1}$ induced from $I_2$. This example motivates the definition of the \emph{twistor space} of a quaternionic manifold $M$ given by Salamon in \cite{S86}: the twistor fibration is the $\cp^1$-bundle $\pi : Z \rightarrow M$ formed by projectivizing the fibres of $\U(M)$ with respect to $I_1$. The total space is referred to as the \emph{twistor space} and the fibres are the \emph{twistor lines}. Since $\U(M)$ is hypercomplex the twistor space is a complex 3-manifold. Also, $I_2$ induces a fixed point free anti-holomorphic involution preserving the twistor lines, which we refer to as the \emph{real structure}.

Quaternionic manifolds of dimension $4k$, where $k > 1$, can be thought of as a generalization of \emph{anti-self-dual} manifolds to higher dimensions, since the construction of twistor spaces given above is a generalization of the construction of twistor spaces for anti-self-dual manifolds \cite{AHS78}. Consequently a quaternionic structure is defined to be equivalent to an anti-self-dual conformal structure.

\subsection{Twistor functions and quaternionic-complex structures}\label{Subsection: Twistor function definition}
Let $M$ be a 4-manifold with an anti-self-dual conformal structure and let $K$ be the canonical bundle of the twistor space $Z$. It was shown in Section 9 of \cite{S82} that the anticanonical bundle admits a unique square root, whose restriction to a twistor line $L$ satisfies $$K^{-1/2}|_L \cong \mathcal{O}(2).$$ Moreover, it was shown that the real structure $\gamma$ lifts to a real structure on $K^{-1/2}$, which we also denote by $\tilde\gamma$. It follows that $\alpha \in H^{0}(Z;K^{-1/2})$ has two zeros on each twistor line and so, it can be written locally as \begin{equation}\label{Equation: Scalar-flat section}\alpha := az_1^2 + bz_1 z_2 + c z_2^2,\end{equation} where $[z_1\!:\!z_2]$ are coordinates on the twistor lines and $a,b,c$ are smooth functions on $M$. Consequently, when $\alpha$ is compatible with the real structure, or equivalently \begin{equation}\label{Equation: Real invariance}\alpha\circ\gamma = \tilde\gamma\circ\alpha,\end{equation} then the zeros of $\alpha$ single out a pair of complex structures that are conjugate with respect to $\gamma$.  This motivates the following well-known theorem concerning scalar-flat K\"ahler representatives of anti-self-dual conformal classes, which was proven by Pontecorvo in \cite{P92}.
\begin{T}\label{Theorem: SFK Penrose transform}
Let $Z$ be the twistor space of an anti-self-dual 4-manifold $M$. The conformal structure of $M$ has a scalar-flat K\"ahler representative if and only if there exists $\alpha \in H^{0}(Z;K^{-1/2})$ that is compatible with the real structure and non-vanishing on each twistor line.
\end{T}
The section $\alpha$ can be pulled back to $\U(M)$ as $$\beta := \pi_2^*{\alpha}.$$ Then the local coordinates $[z_1: z_2]$, with respect to which $\alpha$ was written in (\ref{Equation: Scalar-flat section}), can be pulled back to local coordinates $z_1, z_2$ on the double cover of $\U(M)$, which has fibre $\h^*$. So, locally, $\beta$ can be regarded as a function on the double cover of $\U(M)$. Moreover, this function is quadratic on the fibres, holomorphic with respect to $I_1$ and satisfies the reality condition $$\beta\circ I_2 = \bar{\beta}.$$

We now define $\mu_2 := \tn{Re}(\beta),$ $\mu_3 := \tn{Im}(\beta)$ and suppose that $\mu_1$ satisfies $\tn{d}\mu_1 = -I_3 \tn{d}\mu_2$. Then the reality condition with respect to $I_2$ and the holomorphicity condition with respect to $I_1$ satisfied by $\beta$ imply that \begin{equation}\label{Equation: Holomorphicity} I_1\tn{d}\mu_1 = I_2\tn{d}\mu_2 = I_3\tn{d}\mu_3,\end{equation} where $I_1,I_2$ and $I_3$ act on the $T^*\U(M)$ factor. Thus, using $\beta$, we have obtained a non-vanishing function $\mu = (\mu_1,\mu_2,\mu_3): \U(M)\rightarrow \tn{Im}\h$ that is quadratic on the (double cover of) fibres of $\U(M)$ and satisfies (\ref{Equation: Holomorphicity}). Conversely, $\beta$ can be reconstructed from $\mu$.

\begin{TD}
We refer to a function $\mu = (\mu_1,\mu_2,\mu_3): \U(M)\rightarrow \tn{Im}\h$ that is quadratic on the (double cover of) fibres of $\U(M)$ and satisfies (\ref{Equation: Holomorphicity}) as a \emph{twistor function}. Thus, a scalar-flat K\"ahler structure on an anti-self-dual 4-manifold can be defined as a quaternionic structure equipped with a non-vanishing twistor function.
\end{TD}

If $M$ is a $4k$-dimensional quaternionic manifold, we can similarly define a twistor function on $\U(M)$. Note that a twistor function $\mu$ could alternatively be considered as a section of $H$ in coordinates on $M$, although for practical purposes we will generally consider $\mu$ as a function on $\U(M)$. If this section of $H$ is non-vanishing, then it distinguishes a complex structure on $M$. Thus, we will refer to a quaternionic manifold equipped with a non-vanishing twistor function as a \emph{quaternionic-complex} manifold. For a more complete discussion see Section 7 of \cite{J91}.
\begin{TR}
Equation (\ref{Equation: Holomorphicity}) will be known as the \emph{holomorphicity} condition as it is equivalent to the Cauchy-Riemann conditions for $\mu_2 + i \mu_3$ with respect to $I_1$, $\mu_3 + i \mu_1$ with respect to $I_2$, and $\mu_1 + i \mu_2$ with respect to $I_3$.
\end{TR}

\subsection{The quaternionic quotient}\label{Subsection: Defining quaternionic quotient}
We will describe a group action on a manifold as \emph{locally free} if the stabilizer at any point is finite. In this subsection we define the quaternionic quotient with respect to a locally free group action. Since we do not assume the action is free, these quotients may be orbifolds. We will assume the reader is familiar with orbifolds and so, we only remark that the way in which tensor bundles extend to orbifolds allows all the definitions above to be extended in the usual fashion, cf. \cite{S57}.

Let $M$ be a quaternionic manifold, let $G$ be a connected Lie group acting smoothly and locally freely on $M$ that preserves the quaternionic structure, and let $\g$ be the Lie algebra of $G$. The action of $G$ on $M$ induces an action of $G$ on $\U(M)$. We will denote the map sending $g\in\g$ to the vector field on $\U(M)$ corresponding to the infinitesimal action of $g$ by $$\mb{A}:\g\rightarrow \Gamma(T\U(M)).$$ Let
$$\mu = (\mu_1,\mu_2,\mu_3):\U(M) \rightarrow \g^*\otimes \tn{Im}\h$$
be a $\g^*$-valued twistor function that is $G$-equivariant with respect to the induced action of $G$ on $\U(M)$ and the co-adjoint action on $\g^*$. We define the \emph{transversality condition} at $p\in \U(M)$ as follows: for any non-zero $g\in \g$ there exists $g'\in\g$ such that
\begin{equation}\label{Equation: Transversality}
(I_1\tn{d}{\mu_1}_p(g')).(\mb{A}_p(g))\neq 0.
\end{equation}
This ensures that in a neighbourhood of $p$ the $I_i$-complexification of the $G$-orbits intersect the level sets of $\mu_i$ once transversely.
\begin{TP}[Joyce \cite{J91}]\label{Proposition: Quaternionic quotient}
Let $M$ be a quaternionic manifold and let $G$ be a Lie group preserving the structure. Suppose that there exists a $G$-equivariant twistor function $\mu: \U(M) \rightarrow \g^*\otimes \tn{Im}\h$ that is not everywhere zero. We define three sets $$P := \mu^{-1}(0) - \{0\}\subset \U(M),$$ $$\quad Q := \pi_2(P) \subset Z,$$ $$R := \pi_1(P) \subset M.$$
If $\mu$ satisfies the transversality condition (\ref{Equation: Transversality}) on $P$, then we will refer to $\mu$ as a \emph{quaternionic moment map} for the action of $G$. If $G$ acts smoothly and locally freely on $R$ then the \emph{quaternionic quotient} of $M$ by $G$, which is defined as $R/G$, is an orbifold with a natural quaternionic structure. The twistor space of $R/G$ is given by $Q/G$, and the associated bundle by $P/G$.
\end{TP}

\begin{figure}[htbp]
\begin{diagram}
P \subset \h^k/\{\pm 1\} & \rTo& \U(M)\\
\dTo && \dTo\\
Q \subset \cp^{2k-1} &\rTo & Z\\
\dTo &&\dTo_\pi\\
R \subset \HP^{k-1} & \rTo & M\\
\end{diagram}
\caption{The quaternionic quotient construction}
\end{figure}

\begin{TR}\label{Remark: Defining momemt map}
Let $h$ be a metric on $\U(M)$ and let $\mb{B}\in \g^*\otimes T\U(M)$ be the vector field that contracts with $h$ to give $I_1 \tn{d}\mu_1$. Then the holomorphicity condition (\ref{Equation: Holomorphicity}) becomes $$I_i\tn{d}\mu_i = \mb{B}^\sharp,$$ where $\,{}^\sharp\,$ denotes contraction with $h$, and the transversality condition (\ref{Equation: Transversality}) can be written as
$$h(\mb{A},\mb{B}) \neq 0.$$ Now suppose $h$ is a \emph{hyperK\"ahler} metric with respect to the hypercomplex structure on $\U(M)$ and so, $h$ descends to a \emph{quaternionic-K\"ahler} metric on $M$ \cite{S90}. If $\mb{A} = \mb{B}$, then $\mu$ is the quaternionic-K\"ahler moment map for the action of $G$ on $M$ defined in \cite{GL88}. So, the quaternionic moment map generalizes the quaternionic-K\"ahler moment map by allowing some freedom in the choice of $\mb{B}$. The only restriction on $\mb{B}$ is the transversality condition, which in the quaternionic-K\"ahler case is automatically satisfied as the $I_i$-complexified $G$-orbits meet the $\mu_i$ level sets \emph{orthogonally}.
\end{TR}

Applying Proposition \ref{Proposition: Quaternionic quotient} to the quaternionic-complex context we have the following corollary defining the \emph{quaternionic-complex quotient}:
\begin{TC}\label{Corollary: Quaternionic-complex quotient}
Let $(M,\nu)$ be a quaternionic-complex manifold and let $G$ be a Lie group preserving the structure. Suppose that there exists a quaternionic moment map for $G$ that is not everywhere zero and does not contain $\nu$ in the span of its components. If $G$ acts smoothly and locally freely on $R$, then $\nu$ descends to a quaternionic-complex structure on $R/G$. We will refer to this procedure as taking the quaternionic-complex quotient. When dim$\;M = 4k$ and dim$\;G = k-1$ the quaternionic-complex quotient is a scalar-flat K\"ahler orbifold.
\end{TC}

Note that, if $\nu$ is contained in the span of the components of the moment map at some point, then $\nu$ will descent to a twistor function that vanishes on $R/G$. Therefore, the complex structure on $M$ associated with $\nu$ will not descend to a well-defined complex structure on $R/G$.

\subsection{Compact 4-orbifolds with torus actions}\label{Subsection: Definite Orbifolds}
In this subsection we give a description of a class of 4-orbifolds admitting torus invariant anti-self-dual structures. The framework we use was introduced by Haefliger and Salem in \cite{HS91}, for a more detailed summary than we present here, see \cite{W09a}.

Let $M$ be a compact and simply-connected 4-orbifold admitting a smooth and effective action of a 2-torus, which we will denote by $F$. We will denote the Lie algebra of $F$ by $\mf{f}$ and the lattice of generators by $\Lambda \subset \mf{f}$.  The orbit space of $M$ is topologically a closed disc $N$, the interior of which consists of orbits with trivial stabilizer subgroup. The boundary of $N$ contains isolated orbits, fixed by the action of $F$, which divide $\partial N$ into connected components. The non-fixed orbits in these components are stabilized by an $S^1$-subgroup of $F$, which corresponds to a primitive element in $\Lambda$. At orbifold points, this subgroup lifts to a stabilizer subgroup in the orbifold cover, and the corresponding lattice point lifts to a non-primitive $u \in \Lambda$. Note that $u$ is only determined by the subgroup up to sign. Therefore, we will avoid this ambiguity by restricting to the $180^\circ$ sector in $\mf{f}$ where either $u = (p,0)$ or $(0,1).u > 0$, with respect to a fixed choice of coordinates for $F$.

Between consecutive $F$-fixed points on $\partial N$ we have associated $u \in \Lambda$, and any consecutive lattice points must be linearly independent. Moreover, in order for $M$ to be simply-connected, the set of lattice points must generate $\Lambda$. The boundary has an induced orientation and so, once we have chosen the first fixed point, we can label these pairs consecutively as $u_1, \ldots, u_k$, where $k -2 = b_2(M)$. In this way, a simply-connected 4-orbifold determines the ordered set of combinatorial data $$\s := \{u_1,\ldots,u_k\} \subset \Lambda.$$
Conversely, an ordered set $\s$ generating $\Lambda$, such that the consecutive pairs are linearly independent, can be used to reconstruct a unique simply-connected orbifold. We will denote the compact toric 4-orbifold corresponding to the combinatorial data $\s$ by $\ms$.

Now suppose that rather than spanning $\Lambda$ the combinatorial data $\s$ spans the sublattice $\Lambda_\s$. There is a compact toric 4-orbifold, which we will denote by $\ms$, that has $S^1$-stabilizer subgroups corresponding to $\s$ in the same way as above. However, $\ms$ will no longer be simply-connected, as it is the quotient of a simply-connected 4-orbifold by the finite subgroup of the torus action generated by $\Lambda/\Lambda_\s$. Conversely, the quotient of any simply-connected orbifold by a finite subgroup of the torus action corresponds to some such $\s$.

As shown in Theorem 1 of \cite{W09a}, $\ms$ admits an anti-self-dual conformal structure if and only if the intersection form on $H^2(\ms,\mathbb{Q})$ is \emph{negative-definite}. We will choose the coordinates for the action of $F$ so that $u_1 = (p,0)$. Then, in terms of $\s$, negative-definiteness is equivalent to $u_1, \ldots, u_k$ being the outward pointing normals to a convex polytope \cite{CS06}. We will refer to $\s$ as \emph{convex}, if $\ms$ has a negative-definite intersection form.

\subsection{Toric anti-self-dual 4-orbifolds}
When $\s$ is convex, $\ms$ admits an anti-self-dual conformal structure. This conformal structure induces the conformal structure of an open disc on $N^\circ$. We will denote the $F$-fixed orbit corresponding to $x_i$ by $\zeta_i$, and the ordered set $$\{\zeta_1, \ldots, \zeta_k\} \subset \partial N,$$ by $\mc{R}$. In Theorem A of \cite{W09a} it was shown that the conformal data $\mc{R}$ determines the conformal structure on $\ms$, up to a conformal isometry. Thus, we will denote the anti-self-dual orbifold corresponding to $\mc{R}$ and $\s$ by $\mrs$. Theorem A of \cite{W09a} states that every compact anti-self-dual 4-orbifold with $\chi_{orb} > 0$ whose isometry group contains a 2-torus is conformally equivalent to $\mrs$, for some $\mc{R}$ and convex $\s$. Since $\mrs$ can be obtained from its universal cover by taking the quotient by $\Lambda/\Lambda_\s$ (where $\s$ generates $\Lambda_\s$), Theorem A reduces to the next proposition.

\begin{TP}\label{Proposition: Theorem A}
Suppose $\mrs$ is simply-connected. Then there is a conformal isometry between $\mrs$ and a quaternionic quotient of $\HP^{k}$ by a $(k-1)$-dimensional subtorus of the maximal torus in $Sp(1).Sp(k)$, where $k - 1 = b_2(\mrs)$.
\end{TP}

In order to prove this proposition we use a characterization of $\mrs$ in terms of a meromorphic function on its twistor space, which we state in the next lemma. Before we do this, note that the action of $F$ induces an action of its complexification $F_\C \cong \C^*\times\C^*$ on the twistor space of $\mrs$. (This lemma is a corollary of the proof of Theorem A in \cite{W09a}.)

\begin{TL}\label{Lemma: Meromorphic classification} Let $M$ be a compact anti-self-dual 4-orbifold with $\chi_{orb}(M) > 0$, and suppose that there is an effective action of $F$ preserving the conformal structure. Let $Z$ be the twistor space of $M$ and let $L$ be a real twistor line in $Z$ about which the action of $F_{\C}$ is locally free. An affine coordinate $z$ can be chosen on $L$ such that:
\begin{enumerate}
\item{the real structure is given by $z \mapsto \bar{z}^{-1}$;}
\item{there is a $F_{\C}$-valued holomorphic function on $L  - \{\pm z_1, \ldots, \pm z_k\}$ satisfying $\psi(z).z = -z$ that can be written as $$\psi(z) = \prod_{i=1}^k \big(\frac{z+z_i}{z-z_i}\big)^{v_i}$$ for some $v_1, \ldots, v_k \in \Lambda$;}
\item{and $z_j = e^{i\theta_j}$ for $j = 1, \ldots, k$, where  $$0 = \theta_1 < \theta_2 < \ldots < \theta_k < \pi.$$}
\end{enumerate}
If $v_1 = u_1 + u_k$ and $v_i = u_i - u_{i-1}$ for some $\s$, and there is a conformal isometry mapping the ordered set $\{z_1^2, \ldots, z_k^2\} \subset \{w \in \C : |w| \leq 1\}$ to $\mc{R} \subset N$, then $M$ is conformally equivalent to $\mrs$.\end{TL}

There is a bijective correspondence between $\s$ and the ordered set $$\mc{T} := \{v_1, \ldots, v_k\} \subset \Lambda$$ that was defined in Lemma \ref{Lemma: Meromorphic classification}. This combinatorial data also arises in the construction of $\mrs$ using the techniques of Joyce in \cite{J95}, which we highlight some details of next.

There is a conformal map between $N$ and the closure of $\mc{H}^2 := \{x + iy : y > 0\}$ in $\cp^1$, which maps $\zeta_i$ to $p_i$. With respect to these coordinates on $N^\circ$ there is a family of functions on \begin{equation}\label{Equation: Elementary Joyce solution}f^{p}(x,y) := \Big((x-p)\big((x-p)^2 + y^2\big)^{-1/2}, y\big((x-p)^2 + y^2\big)^{-1/2}\Big),\end{equation} parameterized by $p \in \R\cup \{\infty\}$. These make up a solution to `Joyce's equation' $$\frac{1}{2} \sum_{i=1}^k f^{p_i}\otimes v_i,$$ which can be used to construct a conformal structure on $\mc{H}^2\times F$. The proof of Theorem 3.3.1 of \cite{J95} shows that this conformal structure is well-defined and compactifies to $\mrs$. This involves showing \begin{equation}\label{Equation: Boundary non-degenacy}\frac{\delta}{y}\;\tn{det}\big(\sum_{i=1}^k f^{p_i} \otimes v_i \big)\neq 0\end{equation} is satisfied on $\mc{H}^2$ and along $\partial \mc{H}^2$, where $\delta.\big((x-p_i)^2 + y^2\big)^{-1/2}$ is bounded and non-zero, for $i = 1,\ldots, k$.

\section{Constructing compact anti-self-dual 4-orbifolds}\label{Section: Quotient}
\subsection{Defining a torus action}\label{Subsection: Torus action}
Let $\s \subset \Lambda$ be the combinatorial data associated with a \emph{simply-connected} anti-self-dual orbifold $\mrs$ in the manner described above. We will use this data to define a torus action on $\HP^{k-1}$ in this subsection.

\begin{TD}\label{Definition: Linear map}
The maximal torus in $Sp(k)$, which acts on the right of $\h^k$, will be denoted by $\hat T^k$. We identify $\hat T^k$ with $\R^k/\Z^k$ in such a way that the subspace of $\R^k$ spanned by $e_i$ is the Lie subalgebra of the $S^1$-subgroup in $\hat T^k$ that acts on the $i^{\tn{th}}$ component of $\h^k$. The quaternionic coordinate $q_i$ on $\h^k$ can be written in terms of complex coordinates with respect to $I_1$ as $q_i := x_i + y_i j$.  We will denote the quotient of $\hat T^k$ with respect to the action of $(-1, \ldots, -1)$ by $T^k$. So, $T^k$ acts effectively on the associated bundle of $\HP^{k-1}$, which we denote by $$\U := \h^k/\{\pm (1, \ldots, 1)\}.$$ There is an $(\R^k)^*$-valued vector field, which we denote by $\mb{X}$, sending $v\in\R^k$ to the vector field on $\U$ corresponding to the infinitesimal action of $v$. In the polar coordinates $(r_i, \theta_i)$ for $x_i$ and $(s_i, \phi_i)$ for $y_i$, this can be written explicitly as
$$\mb{X}_q: \R^k \rightarrow T_q\U; \,\, e_i \mapsto \;\partial_{\theta_i} - \partial_{\phi_i},$$
for $q\in \U$ with $x_i, y_i \neq 0$.

The combinatorial data $\mc{T}$ (associated with $\s$) can be used to define a linear map \begin{equation}\label{Equation: Quotient map}\Omega: \R^k \rightarrow \mf{f}; \;\; e_i \mapsto v_i.\end{equation} We will denote the kernel of $\Omega$ by $\mf{g} \subset \R^k$, and corresponding $(k-2)$-dimensional subtorus of $T^k$ by $G_\s$. We will denote by $A: \mf{g} \hookrightarrow \R^k$ the linear map embedding $\mf{g}$ as a Lie subalgebra of $\R^k$. Then, the $\mf{g}^*$-valued vector field that sends $v\in\mf{g}$ to the vector field on $\U$ corresponding to the infinitesimal action of $v$ is given by $$\mb{A} := \mb{X} \circ A \in\Gamma(\mf{g}^*\otimes T\U).$$ \end{TD}

The lattice in $\mf{f}$ defining $T^k$ is generated by $$\{\frac{1}{2}(e_1 + \ldots + e_k), e_2, \ldots, e_k\}.$$ Recall that $v_1 = u_1 + u_k$ and $v_i = u_i - u_{i-1}$, for $i = 2, \ldots, k$ and so, $$u_i = \frac{1}{2}\sum_{j=1}^i v_j - \frac{1}{2}\sum_{j = i+1}^k v_j.$$ From the definition of $G_\s$ in (\ref{Equation: Quotient map}), it follows that $T^k/G_\s$ corresponds to the lattice generated by $u_1, \ldots, u_k$. Since $\ms$ is simply-connected, $u_1, \ldots, u_k$ generate $\Lambda \subset \mf{f}$ and therefore, $T^k/G_\s = F$.

\subsection{Defining a quaternionic moment map}\label{Subsection: Moment map}
Let $\mc{R}\subset \partial N$ be the conformal data associated with $\mrs$ in the manner described above. Identifying $N$ with $\{z^2 : |z| \leq 1 \}$, we will assume that the $F$-fixed orbits $\zeta_1, \ldots, \zeta_k \in \partial N$ satisfy $|\zeta_i| = 1$ and $$0 = \tn{arg}(\zeta_1) < \tn{arg}(\zeta_2) < \ldots < \tn{arg}(\zeta_k) < 2\pi.$$ In this subsection we will use this data to define a quaternionic moment map for the action of $G_\s$ on $\HP^{k-1}$.

We set $z_i = \zeta_i^{1/2}$ with $0 \leq \tn{arg}(z_i) < \pi$, for $i = 1,\ldots, k$, and define $\bs z \in \C^k$ to have $i^{\tn{th}}$ component $z_i$. Then, we define a vector space $$\mathbb{W} := \langle\{\tn{Re}(\bs z), \tn{Im}(\bs z)\}\rangle\subset (\R^k)^*.$$ Let $B^* : (\R^k)^* \rightarrow \mf{g}^*$ be a linear map with kernel $\mathbb{W}$, and let $B: \mf{g}\hookrightarrow \R^k$ be the dual map. We can associate to $B$ a vector field $$\mb{B} := \mb{X}\circ B \in \Gamma(\mf{g}^*\otimes T\U).$$

Recalling Remark \ref{Remark: Defining momemt map}, we can define a quaternionic moment map $\mu_\mc{R} : \U \rightarrow \mf{g}^*\otimes \tn{Im}\h$ using the metric $h$, induced from the Euclidean metric on $\h^k$, by setting $$I_1\tn{d}\mu_1 = I_2\tn{d}\mu_2 = I_3\tn{d}\mu_3 = \mb{B}^\sharp,$$ where $\,{}^\sharp\,$ denotes contraction with $h$. This moment map will be determined uniquely by requiring that $\mu_\mc{R}$ sends $0\in \U$ to $0\in \mf{g}^*\otimes \tn{Im}\h$. Note that the introduction of $h$ is for the sake of convenience in calculations and exposition; the construction is independent of the choice of metric.

The $S^1$-subgroup generated by $e_m \in \R^k$ corresponds to a (quaternionic-K\"ahler) moment map, which can be written as a $T^k$-invariant twistor function on $\HP^{k-1}$: \begin{equation}\label{Equation: Twistors}\begin{array}{ccc} \nu_m:\U \rightarrow \tn{Im}\h;\\ (x_1 + y_1 j,\ldots,x_k + y_k j) \mapsto (|x_m|^2 - |y_m|^2, \textnormal{Re}(2i\;\!x_m\;\!y_m),  \textnormal{Im}(2i\;\!x_m\;\!y_m)).\end{array}\end{equation} Then, $\mu_{\mc{R}}$ can be written in terms of these twistor functions as \begin{equation}\label{Equation: Moment map}\mu_\mc{R} = B^*.\nu,\end{equation} where $\nu := (\nu_1, \ldots,\nu_k):\U \rightarrow (\R^k)^*\!\otimes \tn{Im}\h$.

By construction $\mu_\mc{R}$ satisfies the holomorphicity condition (\ref{Equation: Holomorphicity}). However, strictly speaking, $\mu_\mc{R}$ is only a quaternionic moment map for the action of $G_\s$ when the transversality condition (\ref{Equation: Transversality}) is satisfied on $P = \mu_\mc{R}^{-1}(0)\backslash\{0\}$.
In the next subsection we verify that this is the case.

\subsection{The transversality condition}\label{Subsection: Transversality}

Since $I_1 \textnormal{d}\mu_1 = \mb{B}^{\sharp}$, the transversality condition (\ref{Equation: Transversality}) is equivalent to the non-degeneracy of the bilinear form on $\mf{g}$ given by $$(u,u') \mapsto h(\mb{A} u, \mb{B} u').$$ Define $B^\bot: \U \times \mf{f} \rightarrow \R^k$ to be the linear map $$\delta.\left(\begin{array}{cc} \tn{Re}(z_1)/|q_1|^2 & \tn{Im}(z_1)/|q_1|^2\\ \vdots & \vdots \\ \tn{Re}(z_k)/|q_k|^2 & \tn{Im}(z_k)/|q_k|^2 \end{array}\right)$$ where $\delta = |q_1|^2\ldots|q_k|^2$, and let $\mb{B}^\bot := \mb{X}\circ B^\bot\in\Gamma(\mf{f}^*\otimes T\U)$. From the definition of $\mb{B}$ we have $$g(\mb{B} u, \mb{B}^\bot w) = 0$$ for every $u \in \mf{g}$ and $w \in \mathbb{W}^*$. Similarly, we can define $\mb{A}^\bot\in\Gamma(\mf{f}^*\otimes T\U)$ using $v_1, \ldots, v_k$ so that $$h(\mb{A} u, \mb{A}^\bot v) = 0,$$ for every $u \in \mf{g}$ and $v \in \mf{f}$. It follows from some linear algebra that the transversality condition is equivalent to the non-degeneracy of the bilinear mapping on $\mf{f}\times\mathbb{W}^*$, given by $$(v,w) \mapsto \delta^{-1} h(\mb{A}^\bot v, \mb{B}^\bot w).$$ Thus, (\ref{Equation: Transversality}) can be written explicitly as \begin{equation}\label{Equation: Transversality two}\delta.\tn{det}\Big(\sum_{i=1}^k \frac{\big(\tn{Re}(z_i), \tn{Im}(z_i)\big)\!\otimes v_i}{|q_i|^2} \Big)\neq 0.\end{equation}

\begin{TL}\label{Lemma: Transversality}
The transversality condition (\ref{Equation: Transversality}) is satisfied on $P = \mu_\mc{R}^{-1}(0)\backslash\{0\}$.
\end{TL}
\begin{proof}
If (\ref{Equation: Transversality}) is satisfied at $q \in P$, then it is satisfied on the $T^k \times CO(3)$-orbit of $q$. It follows from equation (\ref{Equation: Moment map}) that $q\in P$ precisely when $\nu(q)$ defines a non-zero vector in $\mathbb{W}\otimes \tn{Im}\h$. By the definition of $\mathbb{W}$, such a vector can be written as $$\tn{Re}(\bs z)\!\otimes \mb{x} + \tn{Im}(\bs z)\!\otimes\mb{y},$$ for $\bs x, \bs y \in \tn{Im}\h$. Therefore, the quotient of $P$ by the action of $T^k$ can be identified with $(\mathbb{W}\otimes \tn{Im}\h) \backslash \{0\}$. The action of $CO(3)$ on $\U$ induces an action of $CO(3)$ on this set and, when $\bs x \neq 0$, a $CO(3)$-orbit contains a unique point with $\bs x = (1,0,0)$ and $\bs y = (-x, y, 0)$, for $y \geq 0$. It follows that there is a bijective correspondence between the $T^k \times CO(3)$ orbits in $P$ and the closure of $\mc{H}^2$ in $\cp^1$, where $\infty$ is identified with the orbit corresponding to $\bs x = 0$.

Note that this quotient of $P$ is the $F$-orbit space of the quaternionic quotient we are constructing, which (as we shall see later) is $\mrs$. We now reduce (\ref{Equation: Transversality}) to a condition on this orbit space, and show that it is equivalent to the restrictions imposed on the solution of Joyce's equation noted in Subsection \ref{Subsection: Definite Orbifolds}. So, we restrict our attention to the points singled out above in each $T^k \times CO(3)$-orbit. A simple calculation using equation (\ref{Equation: Twistors}) shows that, if $\nu(q) = \tn{Re}(\bs z)\!\otimes \mb{x} + \tn{Im}(\bs z)\!\otimes\mb{y}$, then $$|q_i|^2 = |\tn{Re}(z_i) \mb{x} +\tn{Im}(z_i) \mb{y}|,$$ for $i = 1,\ldots, k$. If we set $p_i = \tn{Re}(z_i)/\tn{Im}(z_i)$, then, after a little rearrangement, (\ref{Equation: Transversality two}) is equivalent to $$\delta.\tn{det}\Big(\sum_{i=1}^k \big(\tn{Re}(f^{p_i}), \tn{Im}(f^{p_i})\big)\otimes v_i \Big)\neq 0,$$ where $f^{p}$ was defined in (\ref{Equation: Elementary Joyce solution}). Therefore, this expression is equivalent to (\ref{Equation: Boundary non-degenacy}) for the conformal data defined by $\infty = p_1 > p_2 > \ldots > p_k$, and (\ref{Equation: Boundary non-degenacy}) is satisfied because $\s$ is convex.
\end{proof}

\subsection{The quaternionic quotient}\label{Subsection: Quaternionic quotient}
Recall from Proposition \ref{Proposition: Quaternionic quotient} that the quaternionic quotient of $\HP^{k-1}$ with respect to $G_\s$ is defined as $$M : = R/G_\s,$$ where $$R = \{[q]\in\HP^{k-1}: q \in P\}.$$
The second part of the next lemma proves that $R/G_\s$ is an orbifold.
\begin{TL}\label{Lemma: Locally Free}
The action of $T^k$ at $[q] \in R$ is locally free if $\nu(q)$ cannot be written as $$\bs w\otimes\mb{x}  \in \mathbb{W}\otimes \tn{Im}\h.$$ Furthermore, the action of $G_\s$ on $R$ is locally free.
\end{TL}
\begin{proof}
Define $G_{\mb{d}}\subset T^k$ to be the $S^1$-subgroup of $G_\s$ acting on $\HP^{k-1}$ as
$$\lambda: [q_1: \ldots: q_k] \mapsto [q_1.\lambda^{d_1}: \ldots: q_k.\lambda^{d_k}],$$
for some non-zero vector $\mb{d} = (d_1, \ldots, d_k)\in \Z^k$.
Suppose $[q] \in R$ is fixed by $G_\mb{d}$. Then, for every $\lambda \in U(1)$, there exists some $c\in CO(3)$ such that
$$(c.q_1,\ldots, c.q_k) = (q_1.\lambda^{d_1}, \ldots, q_k.\lambda^{d_k}) \in \U.$$
We will assume that $q_1 \neq 0$. Then, using the action of $CO(3)$ we may set $q_1 = 1$, in which case $c = \lambda^{d_1}$. (If $q_1 = 0$, then we can set some other $q_i = 1$, and the argument proceeds similarly.) Consequently, for $i=2,\ldots,k$, either \begin{itemize}\item{$q_i \in \C$ and $d_i = d_1$,} \item{$-q_ij\in \C$ and $d_i = -d_1$,} \item{or $q_i = 0$.}\end{itemize} Recall that $\nu(q)$ defines a non-zero vector in $\mathbb{W}\otimes \tn{Im}\h$ when $q\in P$. So, the conditions imposed on $q$ imply that $\nu(q) = \bs w\otimes (1,0,0)$, for some $\bs w \in \mathbb{W}$. Moreover, the induced action of $CO(3)$ sends $\bs w\otimes (1,0,0)$ to $\bs w\otimes\mb{x}$, for some non-zero $x \in \tn{Im}\h$. This completes the proof of the first statement in the lemma.

Now suppose that $[q]$ is fixed by an $S^1$-subgroup of $G_\s$. From our definition of $G_\s$ it follows that $G_\mb{d}\subset G_\s$ precisely when
\begin{equation}\label{Equation: Circle subgroup}\sum_{i=1}^k d_i v_i = 0.\end{equation} The condition  $\nu(q) = \bs w\otimes (1,0,0)$, along with the definition of $\mathbb{W}$ in Subsection \ref{Subsection: Moment map}, implies that there is some $m\in \{2,\ldots, k\}$ such that
$$[q_1: \ldots : q_k] = [1:x_2:\ldots:x_m:y_{m+1}j:\ldots:y_kj],$$
where $x_m\in \C$ and otherwise the entries are in $\C^*$.
Thus, we find that $\mb{d}$ takes the form
$$(d_1, \ldots, d_1, d_m, -d_1, \ldots, -d_1).$$
If we substitute $\mb{d}$ into equation (\ref{Equation: Circle subgroup}), it follows that $$(d_1 + d_m)u_m + (d_1 - d_m)u_{m-1} = 0,$$ since $v_1 = u_1 + u_k$ and $v_i = u_i - u_{i-1}$, for $i = 1,\ldots, k$.
Therefore, $u_m$ and $u_{m-1}$ are linearly dependent. This gives a contradiction, as noted in Subsection \ref{Subsection: Definite Orbifolds}. \end{proof}

From Lemma \ref{Lemma: Transversality} and Lemma \ref{Lemma: Locally Free} we see that $M$ is a quaternionic orbifold, and a dimension count shows us that $M$ is 4 dimensional. Also $M$ is compact, since $\HP^{k-1}$ is compact. The twistor functions making up the components of $\mu_\mc{R}$ are $T^k$-invariant and so, this quaternionic quotient construction is $T^k$-invariant. Since the action of $T^k$ preserves the quaternionic structure on $\HP^{k-1}$, it follows that the (effective) action of $F = T^k/G$ on $M$ preserves the quaternionic structure. Therefore, $M$ is a compact anti-self-dual 4-orbifold whose isometry group contains $F$.

Finally, we need to verify that $\chi_{orb}(M) > 0$ in order to show that $M$ belongs to the class we are concerned with. By the Poincar\'e-Hopf Theorem for orbifolds, $\chi_{orb}(M) > 0$ is equivalent to the existence of $F$-fixed points (see for example Subsection 2.2 of \cite{W09a}). The existence of such points is clear: the sets $\{[q]\in R: q_i = 0\}\subset \HP^{k-1}$ are mapped to fixed points of the $F$-action in $M$ under the quotient with respect to $G_\s$.

\begin{TP}\label{Proposition: ASD orbifold}
Suppose the anti-self-dual orbifold $\mrs$ is simply-connected. Then the quaternionic quotient of $\HP^{k-1}$ with respect to $G_\s$ and $\mu_\mc{R}$ is a compact anti-self-dual 4-orbifold $M$ with $\chi_{orb}(M) > 0$. Furthermore, $F$ acts effectively on $M$ and preserves the conformal structure. \end{TP}

\section{The proofs of Theorems A and B}\label{Section: Proofs of Theorem A and B}
\subsection{A meromorphic function in higher dimensions}\label{Subsection: Higher dim meromorphic}
In order to prove Theorem A we now look at an example of a meromorphic function on a twistor line in $\cp^{2k-1}$, which is defined similarly to $\psi$ in Lemma \ref{Lemma: Meromorphic classification}. Recall from Proposition \ref{Proposition: Quaternionic quotient} that the twistor space of $M$ is given by $$Z := Q/G_\s,$$ where $$Q := \{[q]\in\cp^{2k-1}: q\in P\subset \U\}.$$ Using equations (\ref{Equation: Twistors}) and (\ref{Equation: Moment map}), $$\nu(q) = 2(0, -\tn{Im}(\bs z) , \tn{Re}(\bs z)),$$ when the components of $q$ are determined by $x_j = 1$, $y_j = z_j$. By the definition of $P$ in Subsection \ref{Subsection: Moment map}, $q \in P$, since $\nu(q) \in \mathbb{W}\otimes \tn{Im}\h$. Therefore, $[q]\in Q$.

Let $\ell$ be the twistor line in $Q\subset \cp^{2k-1}$ that contains $[q]$. Then, by Lemma \ref{Lemma: Locally Free}, the action of $T^k_\C$ is locally free about $\ell$. Recalling the description of a twistor space in Subsection \ref{Subsection: Hypercomplex and quaternionic} it follows that $\ell$ consist of the points $$[x_1x - \bar{y}_1 y: y_1 x + \bar{x}_1 y: \ldots : x_k x - \bar{y}_k y: y_k x + \bar{x}_k y],$$ for $[x:y]\in\cp^1$. We will use the coordinate $z = y/x$ on $\ell$, and note that with this coordinate the real structure, which is induced from multiplication by $j$, is the antipodal map. Then, we define $$\ell' : = \ell - \{\pm z_1,\ldots, \pm z_k\}.$$

The action of $T^k$ on $\U$ descends to an action of $T^k$ on $\cp^{2k-1}$. This action extends to an action of the complex torus $T^k_\C$ in the obvious way
$$(\lambda_1, \ldots, \lambda_k): [x_1: y_1:\ldots:x_k: y_k] \mapsto [\lambda_1 x_1: \lambda_1^{-1}y_1:\ldots:\lambda_k x_k: \lambda_k^{-1} y_k],$$ where $(\lambda_1, \ldots, \lambda_k)$ is identified with $(-\lambda_1, \ldots,-\lambda_k)$ in $T^k_\C$. An elementary calculation then verifies the next lemma.

\begin{TL}\label{Lemma: Chi}
There is a non-trivial $T^k_{\C}$-valued holomorphic function on $\ell'$, which we will denote by $\Psi$, satisfying $\Psi(z).z = -z \in \ell'$. This function can be written as $$\Psi: z \mapsto \pm \Big(\frac{z + z_1}{z - z_1}, \ldots, \frac{z + z_k}{z-z_k} \Big).$$
\end{TL}

\subsection{The proof of Theorem A}\label{Subsection: Theorem A}
The next lemma completes the proof of Proposition \ref{Proposition: Theorem A} and therefore, Theorem A.

\begin{TL}
There is a conformal isometry between $M$ and $M_{\s,\mc{R}}$.
\end{TL}
\begin{proof}
First note that by Proposition \ref{Proposition: ASD orbifold}, $M$ is a compact anti-self-dual 4-orbifold with $\chi_{orb}(M) > 0$ whose isometry group contains $F$. Also, by construction, $F$ act effectively on $M$. When we take the quotient of $Q$ with respect to $G_\s$, the image of $\ell$ defines a twistor line in $Z$, which we will denote by $L$. As noted above the action of $T^k_\C$ about $\ell$ is locally free; thus, the action of $F_\C$ about $L$ is locally free. Therefore, we can apply Lemma \ref{Lemma: Meromorphic classification}, once we have found an appropriate affine coordinate on $L$.

We will use the coordinate on $L$ induced from $\ell$; thus, the real structure on $Z$ restricts to $z \mapsto \bar{z}^{-1}$ on $L$. This satisfies part (i) of Lemma \ref{Lemma: Meromorphic classification}. The $T^k_\C$-valued function $\Psi$ defined in Lemma \ref{Lemma: Chi} descends to an $F_\C$-valued holomorphic function on $L - \{\pm z_1,\ldots, \pm z_k\}$ satisfying $\psi(z).z = -z.$ By equation (\ref{Equation: Quotient map}) for the quotient map from $T^k$ to $F$, $$\psi(z) = \Big(\frac{z+z_i}{z-z_i}\Big)^{v_i}.$$ This satisfies part (ii) of Lemma \ref{Lemma: Meromorphic classification}.

By construction, $\{z_1, \ldots, z_k\}$ satisfies part (iii) of Lemma \ref{Lemma: Meromorphic classification} and there is a conformal isometry of the unit disc mapping $\{z_1^2, \ldots, z_k^2\}$ to $\mc{R}$. Therefore, $M$ is conformally equivalent to $\mrs$.
\end{proof}

\subsection{The proof of Theorem B}\label{Subsection: Theorem B}
We will denote by $Y$ the complement of the $F$-fixed point $x_1$ in $\mrs$. In the quaternionic quotient construction $x_1$ corresponds to $$\{[q_1:\ldots:q_k]\in \HP^{k-1}: q_1 = 0\}.$$ Thus, $Y$ can be constructed by taking the quaternionic quotient of $\{q_1\neq 0\} \subset \HP^{k-1}$ with respect to $G_\s$ and $\mu_\mc{R}$. As explained in Subsection \ref{Subsection: Twistor function definition}, the scalar-flat K\"ahler structures on $Y$ that are invariant under the action of $F$ correspond to $T^k$-invariant twistor functions that are non-vanishing where $q_1\neq 0$. The space of all $T^k$-invariant twistor functions on $\HP^{k-1}$ is spanned by $\{\nu_1,\ldots,\nu_k \}$, where the $\nu_i: \U \rightarrow \tn{Im}\h$ were defined in equation (\ref{Equation: Twistors}). The only twistor function in this space that is non-vanishing away from $\{q_1 = 0\}$ is $\nu_1$. Since $\nu_1$ is not in the span of the components of $\mu_\mc{R}$ for any choice of $\mc{R}$, Corollary \ref{Corollary: Quaternionic-complex quotient} can be applied to find a unique $F$-invariant scalar-flat K\"ahler representative of the conformal structure on $Y$.

So let $X$ be a scalar-flat K\"ahler 4-orbifold whose isometry group contains a 2-torus, which is ALE to order $l > 3/2$. Theorem B in \cite{W09a} states that there is an $F$-equivariant isometry between $X$ and the scalar-flat K\"ahler representative in $Y$, for some $\mc{R}$ and $\s$. It follows that $X$ is isometric to the quaternionic-complex quotient of $(\{q_1 = 0\}, \nu_1)$ with respect to $G_\s$ and $\mu_\mc{R}$. This completes the proof of Theorem B.

\begin{TR}
To construct ALE hyperK\"ahler metrics on the minimal resolution of $\C^2/\Gamma$, for $\Gamma \subset SU(2)$, Kronheimer uses the irreducible representations of $\Gamma$ on $\C^2$ to define a quaternionic vector space and group action, with respect to which the hyperK\"ahler quotient can be taken \cite{K89a}. While the hyperK\"ahler quotient provides an exhaustive list of ALE hyperK\"ahler 4-manifolds \cite{K89b}, the same could not be achieved using the quaternionic quotient, since there are examples of ALE scalar-flat K\"ahler 4-manifolds that do not compactify to anti-self-dual orbifolds with Moishezon twistor spaces \cite{LP92}. Despite this, we would expect that more examples of ALE scalar-flat K\"ahler metrics on resolutions of $\C^2/\Gamma$, for $\Gamma\subset U(2)$, could be constructed as quaternionic-complex quotients of quaternionic vector spaces. It may be possible to do this using a generalization of Kronheimer's representation theoretic approach or alternatively by using Joyce's procedure for gluing copies of the weighted projective plane using the quaternionic quotient \cite{J92}.
\end{TR}

\subsection{Non-toric deformations}\label{Subsection: Non-toric deformations}
We conclude with a remark about deformations that arise within the framework of the quaternionic quotient. The conformal structures that we have constructed so far are toric, since the moment maps are invariant under the action of $T^k$. In general, for $\s$ convex, the space of $T^k$-invariant twistor functions coincides with the space of $G_\s$-invariant twistor functions and so, the family $\mrs$ provides all the quaternionic quotients of $\HP^{k-1}$ by $G_\s$. However, in some special cases there are $G_\s$-invariant twistor functions that are not $T^k$-invariant. We will refer to the convex data $\s$ as \emph{deformable} if the space of $G_\s$-invariant twistor functions is larger than the space of $T^k$-invariant twistor functions on $\HP^{k-1}$. In \cite{J92} Joyce determines which $\s$ are deformable using some elementary linear algebra.

When $\s$ is deformable, the additional twistor functions can be included in the components of the moment map for the action of $G_\s$, and the moment map will still satisfy the transversality condition. Consequently, this gives a larger family of anti-self-dual structures on the underlying orbifold, more precisely when $k > 4$ the $(k-3)$-dimensional torus invariant family is contained in a $3(k-4)$-dimensional family of anti-self-dual structures. It can be easily verified that LeBrun's anti-self-dual metrics on $(k-2)\overline{\cp}^2$ are the only \emph{smooth} examples that can be deformed in this way.

Also note that non-toric ALE scalar-flat K\"ahler metric on the complement of a point in $\ms$ can be constructed in the manner described above, using a deformable $\s$. In particular: the non-toric examples of Kronheimer's construction of hyperK\"ahler metrics on the minimal resolution of a type-A singularity \cite{K89a}, and LeBrun's family of scalar-flat K\"ahler metrics on $\C^2$ blown-up at $k-2$ points \cite{L91}.

\end{document}